\documentclass[12pt,a4paper,reqno]{amsart}
\usepackage[top=2.65cm, bottom=2.65cm, left=2.5cm, right=2.5cm]{geometry}
\usepackage[T1]{fontenc}
\usepackage[utf8]{inputenc}
\usepackage{amssymb}
\usepackage{amsthm}
\usepackage{graphics}
\usepackage{mathtools}
\usepackage{tikz}
\usetikzlibrary{patterns}

\makeatletter
\@namedef{subjclassname@2020}{%
  \textup{2020} Mathematics Subject Classification}
\makeatother

\theoremstyle{plain}
\newtheorem{theorem}{Theorem}
\newtheorem{corollary}[theorem]{Corollary}
\newtheorem{lemma}[theorem]{Lemma}
\newtheorem{proposition}[theorem]{Proposition}
\newtheorem{conjecture}[theorem]{Conjecture}
\theoremstyle{definition}
\newtheorem{definition}[theorem]{Definition}
\newtheorem{example}[theorem]{Example}
\newtheorem{remark}[theorem]{Remark}

\newcommand{\N}{{\mathbb N}}
\newcommand{\wasc}{\mathrm{wasc}}
\newcommand{\asc}{\mathrm{asc}}
\newcommand{\mynewpage}{}
\newcommand{\Wasc}{\mathrm{WAsc}}
\newcommand{\newmatrices}{\mathrm{WMat}}
\newcommand{\sn}{S_n}
\newcommand{\MM}{\Omega}
\newcommand{\ptm}{\Phi}
\newcommand{\PosetMap}{\Psi}

\newcommand{\Fish}{\mathcal{F}}
\newcommand{\RR}{\mathcal{W}}
\newcommand{\LL}{\Gamma}
\newcommand{\sss}{\mathrm{numact}}
\newcommand{\bbb}{\mathrm{lastact}}
\newcommand{\fishpattern}{F}
\newcommand{\wpattern}{W}

\newcommand{\Mat}[1]{\left[\begin{smallmatrix}#1\end{smallmatrix}\right]}

\newcommand{\spone}{\mathrm{topone}}
\newcommand{\reduce}{\mathrm{reduce}}
\newcommand{\madd}{\mathrm{expand}}
\newcommand{\WPoset}{\mathrm{WPoset}}
\newcommand{\desbot}{\mathrm{desbot}}
\newcommand{\emptyword}{\epsilon}

\newcommand{\pattern}[4]{
  \raisebox{0.6ex}{
  \begin{tikzpicture}[scale=0.35, baseline=(current bounding box.center), #1]
    \foreach \x/\y in {#4}
      \fill[pattern=north east lines, pattern color=black!45] (\x,\y) rectangle +(1,1);
    \draw (0.01,0.01) grid (#2+0.99,#2+0.99);
    \foreach \x/\y in {#3}
      \filldraw (\x,\y) circle (5pt);
  \end{tikzpicture}}\;
}
\newcommand{\scl}{1}
\newcommand{\nyip}{\hspace*{0.425em}}

\def\lgrey{0.9}
\def\softgrey{0.8}
\definecolor{lightgrey}{rgb}{\lgrey,\lgrey,\lgrey}
\definecolor{grey}{rgb}{\softgrey,\softgrey,\softgrey}
\newcommand{\myonematrix}{
    \def\mheight{5}
    \def\xdiff{0.5}
    \def\ydiff{1}
    \def\sdiff{0.1}
    \def\mi{2}
    \def\mj{4}
    \def\mip{4}
    \def\mjp{2.25}
    \def\yloola{0.5}
    \def\xloola{0.75}
    \def\mskip{0.25}
    \centerline{ \begin{tikzpicture}[scale=0.5]
    \tikzstyle{disc} = [ rectangle,fill=black,draw=black, minimum size=3.6pt, inner sep=0pt];
    \tikzstyle{wdisc} = [ rectangle,fill=white,draw=black, minimum size=3.6pt, inner sep=0pt];
    \tikzstyle{gdisc} = [ rectangle,fill=white,draw=grey, minimum size=3.6pt, inner sep=0pt];
    \draw[line width=1] (0,0) .. controls (-\xdiff,\ydiff) and (-\xdiff,\mheight-\ydiff) .. (0,\mheight);
    \draw[line width=1] (\mheight,0) .. controls (\mheight+\xdiff,\ydiff) and (\mheight+\xdiff,\mheight-\ydiff) .. (\mheight,\mheight);
    \draw node [disc] (ij) at (\mi,\mj) {};
    \draw node [disc] (ipjp) at (\mip,\mjp) {};
    \draw node [wdisc] (ipj) at (\mi,\mjp) {};
	\node [above=2pt] at (ij) {$i$};
	\node [right=2pt] at (ipjp) {$j$};
    \foreach \x in {1,2,...,19}
        \draw node [gdisc]  at ( 0 + \mskip*\x , \mheight - \mskip*\x ) {};
    \draw (ij) -- (ipj);
    \draw (ipjp) -- (ipj);
\end{tikzpicture} }
}

\title[Weak ascent sequences]{Weak ascent sequences and related\\ combinatorial structures}
\author{Be\'ata B\'enyi}
\address{B. B\'enyi: Faculty of Water Sciences, University of Public Service, Baja, Hungary}

\author{Anders Claesson}
\address{A. Claesson: Science Institute, University of Iceland, Iceland}

\author{Mark Dukes}
\address{M. Dukes: School of Mathematics \& Statistics, University College Dublin, Ireland}

\keywords{Weak ascent sequence; (2+2)-free poset, Bivincular pattern, Pattern avoiding permutation, Enumeration.}
\subjclass[2020]{05A05,05A19,06A11}

\begin{document}
\begin{abstract}
  In this paper we introduce weak ascent sequences, a class of
  number sequences that properly contains ascent sequences.  We show
  how these sequences uniquely encode each of the following objects:
  permutations avoiding a particular length-4 bivincular pattern;
  upper-triangular binary matrices that satisfy a column-adjacency
  rule; factorial posets that are weakly (3+1)-free.
  We also show how weak ascent sequences are related to a class of
  pattern avoiding inversion sequences that has been a topic of recent
  research by Auli and Elizalde. Finally, we consider the problem of
  enumerating these new sequences and give a closed form expression
  for the number of weak ascent sequences having a prescribed length
  and number of weak ascents.
\end{abstract}
\maketitle
\thispagestyle{empty}

\section{Introduction}
Ascent sequences~\cite{BMCDK} are rich number sequences in that they
uniquely encode four different combinatorial objects and thereby
induce bijections between these objects.  These objects are (2+2)-free
posets; Fishburn permutations;
upper-triangular matrices of non-negative integers having neither
columns nor rows of only zeros; and Stoimenow matchings.  Statistics on
those objects have been shown to be related to natural considerations on
the ascent sequences to which they correspond.

In this paper we will define a new sequence that we term a {\it{weak
    ascent}} sequence and study the rich connections these sequences
have to other combinatorial objects
that are similar in spirit to those mentioned above.
Given a sequence of integers
$x=(x_1,\ldots,x_n)$, we say there is a {\it{weak ascent}} at position
$i$ if $x_i \leq x_{i+1}$.  We denote by $\wasc(x)$ the number of weak
ascents in the sequence $x$.
Throughout this paper we will use the notation $[a,b]$ for the set $\{a,a+1,a+2,\ldots,b\}$.

\begin{definition}\label{was-defn}
  We call a sequence of integers $a=(a_1,\ldots,a_n)$ a {\textit{weak
      ascent sequence}} if $a_1=0$ and
  $a_{i+1} \in [0,1+\wasc(a_1,\ldots,a_i)]$ for all $i\in [0,n-1]$.
  Let $\Wasc_n$ be the set of weak ascent sequences of length $n$.
\end{definition}

In Table~\ref{allfourseqs} we list all weak ascent sequences of length
at most four.

\begin{table}[!ht]
  $$\def\arraystretch{1.2}
  \begin{array}{c@{\quad}l}\hline
   n & \Wasc_n \\ \hline \\[-3ex]
   1 & (0) \\
   2 & (0,0), \, (0,1) \\
   3 & (0,0,0), \, (0,0,1),\, (0,0,2),\, (0,1,0),\, (0,1,1),\, (0,1,2) \\
   4 & (0,0,0,0),\, (0,0,0,1),\, (0,0,0,2),\, (0,0,0,3),\, (0,0,1,0),\, (0,0,1,1),\, (0,0,1,2), \\ &
       (0,0,1,3),\, (0,0,2,0),\, (0,0,2,1),\, (0,0,2,2),\, (0,0,2,3),\, (0,1,0,0),\, (0,1,0,1), \\ &
       (0,1,0,2),\, (0,1,1,0),\, (0,1,1,1),\, (0,1,1,2),\, (0,1,1,3),\, (0,1,2,0),\, (0,1,2,1), \\ &
       (0,1,2,2),\, (0,1,2,3) \\[1ex] \hline
  \end{array}
  $$
\caption{All weak ascent sequences of length at most 4.\label{allfourseqs}}
\end{table}

To contrast this with the original ascent sequences, recall that a
sequence of integers $x=(x_1,\ldots,x_n)$ has an {\it{ascent}} at
position $i$ if $x_i < x_{i+1}$.  An \emph{ascent sequence} is a
sequence of integers $a=(a_1,\ldots,a_n)$ with $a_1=0$ and
$a_{i+1} \in [0,1+\asc(a_1,\ldots,a_i)]$ for all $i\in [0,n-1]$, where
$\asc$ denotes the number of ascents in the sequence. Clearly, ascent
sequences are weak ascent sequences.

In this paper we will show how these weak ascent sequences uniquely
encode each of the following objects: permutations avoiding a
particular length-4 bivincular pattern (in
Section~\ref{permssection}); upper-triangular binary matrices that
satisfy a column-adjacency rule (in Section~\ref{matricessection});
factorial posets that contain no {\em weak} (3+1) subposet (in
Section~\ref{posetssection}).  We show in
Section~\ref{inversionssection} how weak ascent sequences are related
to a class of pattern avoiding inversion sequences that has been a
topic of recent research by Auli and
Elizalde~\cite{AuliSergi1,AuliSergi2,elizalde}.  In
Section~\ref{inversionssection} we also consider the problem of
enumerating these new sequences and give a closed form expression for
the number of weak ascent sequences having a prescribed length and
number of weak ascents.

The objects that we study in this paper are summarised in
Figure~\ref{fig:summary}, along with the names of the bijections that
we construct and prove between these objects.  In that diagram we also
include the section numbers where each of the bijections may be found.

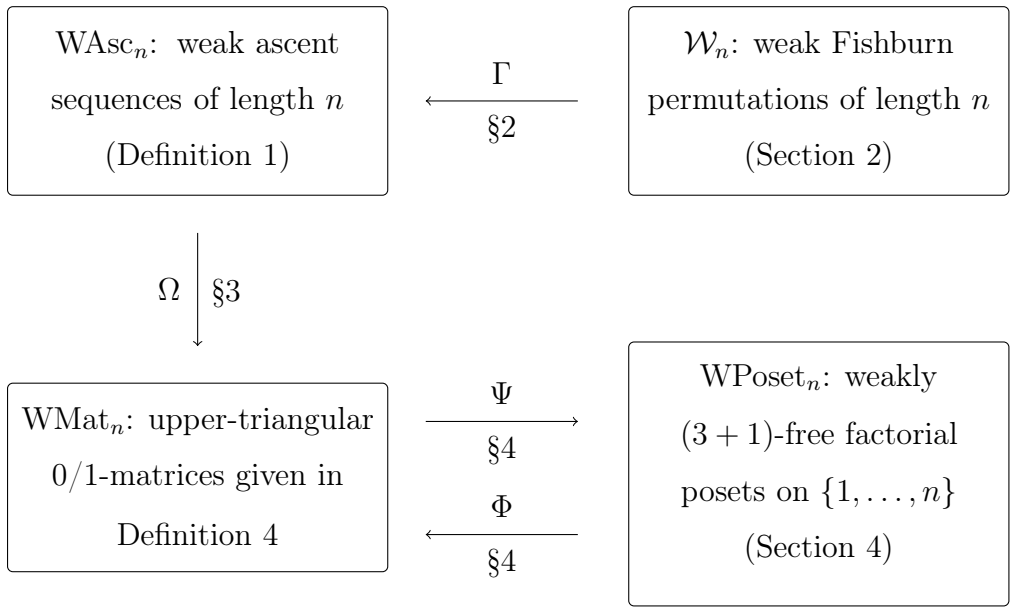
\begin{figure}[!h]
\begin{tikzpicture}[scale=0.5, rounded corners = 2pt]
\begin{scope}
\draw (0,4) node [align=left] {$\Wasc_n$:\, weak ascent};
\draw (0,2.5) node [align=left] {sequences of length $n$};
\draw (0,1) node [align=left] {(Definition~\ref{was-defn})};
\draw (-5,0) rectangle (5,5);
\end{scope}
\begin{scope}[xshift = 90ex]
\draw (0,4) node [align=left] {$\RR_n$: weak Fishburn};
\draw (0,2.5) node [align=left] {permutations of length $n$};
\draw (0,1) node [align=left] {(Section~\ref{permssection})};
\draw (-5,0) rectangle (5,5);
\end{scope}
\begin{scope}[yshift = -55ex]
\draw (0,4) node [align=left]  {$\newmatrices_n$: upper-triangular};
\draw (0,2.5) node [align=left] {0/1-matrices given in};
\draw (0,1) node [align=left] {Definition~\ref{matrixdefn}};
\draw (-5,0) rectangle (5,5);
\end{scope}
\begin{scope}[xshift = 90ex, yshift = -60ex]
\draw (0,6.05) node {$\WPoset_n$: weakly};
\draw (0,4.5) node {$(3+1)$-free factorial};
\draw (0,3.00) node {posets on $\{1,\ldots,n\}$};
\draw (0,1.5) node {(Section~\ref{posetssection})};
\draw (-5,0) rectangle (5,7);
\end{scope}
\draw[->] (0,-1) -- (0,-4);
\draw (-0.75,-2.5) node {$\MM$};
\draw (0.75,-2.6) node {\S3};
\draw[<-] (6,2.5) -- (10,2.5);
\draw (8, 3.25) node {$\LL$};
\draw (8, 1.75) node {\S2};
\draw[->] (6,-6) -- (10,-6);
\draw (8, -5.25) node {$\PosetMap$};
\draw (8, -6.75) node {\S\ref{posetssection}};
\draw[<-] (6,-9) -- (10,-9);
\draw (8, -8.25) node {$\ptm$};
\draw (8, -9.75) node {\S\ref{posetssection}};
\end{tikzpicture}
\caption{Diagrammatic summary of the sets and bijections of interest.}
\label{fig:summary}
\end{figure}

\mynewpage
\section{Weak Fishburn permutations}
\label{permssection}

Let $\sn$ be the set of permutations of the set $\{1,\ldots,n\}$.  Given
a pattern $P$, in the pattern-avoidance literature the convention is to
denote by $\sn(P)$ the set of permutations in $\sn$ that do not contain
the pattern $P$. The set of Fishburn permutations~\cite{BMCDK,
  GilWeiner}, $\Fish_n=\sn(\fishpattern)$, are those that avoid the bivincular
pattern
$$
\fishpattern = \big(\,231,\,[0,3]\!\times\!\{1\}\cup \{1\}\!\times\![0,3]\,\big)
= \pattern{scale=\scl}{3}
   {1/2, 2/3, 3/1}
   {0/1, 1/1, 2/1, 3/1, 1/0, 1/2, 1/3},
$$
here defined and depicted as a mesh pattern~\cite{anderspetter}.  The
inclusion of shaded rows and columns indicates that in an occurrence
of such a pattern in a permutation, there should be no other
permutation dots
in the shaded zones
when this pattern is placed over a
permutation. Bousquet-M\'elou et al.~\cite{BMCDK} gave a bijection
between ascent sequences and Fishburn permutations. More precisely,
ascent sequences encode the so called active sites of the
Fishburn permutations.

We define the bivincular pattern
$$
\wpattern = \big(\,3412,\,[0,4]\!\times\!\{2\}\cup \{1\}\!\times\![0,4]\,\big)
= \pattern{scale=\scl}{4}
   {1/3, 2/4, 3/1, 4/2}
   {1/0, 1/1, 1/2, 1/3,1/4, 0/2, 1/2, 2/2, 3/2, 4/2}
$$
and call $\RR_n=\sn(\wpattern)$ the set of \emph{weak Fishburn
  permutations}.  For the benefit of readers not familiar with
bivincular or mesh patterns we also give an elementary definition of
the \emph{weak Fishburn pattern} $\wpattern$: A permutation
$\pi\in\sn$ contains $\wpattern$ if there are four indices
$1\leq i<j<k<\ell \leq n$ such that $j= i+1$, $\pi_i = \pi_{\ell} +1$
and $\pi_k < \pi_{\ell} < \pi_i < \pi_j$. In this case we also say
that $\pi_i\pi_j\pi_k\pi_{\ell}$ is an \emph{occurrence} of
$\wpattern$ in $\pi$. If there are no occurrences of $\wpattern$ in
$\pi$, then we say that $\pi$ \emph{avoids} $\wpattern$.

If $\pi_i\pi_j\pi_k\pi_{\ell}$ is an occurrence of $\wpattern$ then
$\pi_i\pi_j\pi_{\ell}$ is an occurrence of $\fishpattern$. In other
words, every Fishburn permutation is a weak Fishburn permutation and
we have $\Fish_n\subseteq\RR_n$.

We can construct permutations in $\RR_n$ inductively: Let $\pi$ be a
permutation in $\RR_{n}$ with $n>0$.  Suppose that $\tau$ is obtained
from $\pi$ by deleting the entry $n$. Then $\tau \in \RR_{n-1}$. To
see why this must be the case, let
$\tau_i \tau_{i+1} \tau_k \tau_{\ell}$ be an occurrence of $\wpattern$
in $\tau$ but not in $\pi$. This can only happen if
$\pi_{i+1}=n$. However, this implies that
$\pi_i \pi_{i+1} \pi_{k+1}\pi_{\ell +1}$ is an occurrence of a
$\wpattern$ in $\pi$.

Given $\tau\in \RR_{n-1}$, let us call the
sites where the new maximal value $n$ can be inserted in $\tau$ so as
to produce an element of $\RR_{n}$ {\em{active sites}}.  The site
before $\tau_1$ and the site after $\tau_{n-1}$ are always active.
Determining whether the site between entries $\tau_i$ and $\tau_{i+1}$
is {active} depends on whether $\tau_i\leq 2$ or if there does not
exist $(\tau_i,t,\tau_i -1)$ in $\tau$ with $t<\tau_i - 1$.  This
latter (non-existence) condition is somewhat hard to absorb, so let us
disentangle it as follows.

The site between entries $\tau_i$ and $\tau_{i+1}$ is active if
\begin{itemize}
\item $\tau_i \leq 2$, or
\item $\tau_i -1 $ is to the left of $\tau_i$, or
\item $\tau_i -1$ is to the right of $\tau_i$ and there is no value
  $t<\tau_i -1$ between $\tau_i$ and $\tau_i -1$.
\end{itemize}

With this notion of active sites let us label the active sites, from
left to right, with $\{0,1,2,\dots\}$.

We will now introduce a map $\LL$ from $\RR_{n}$ to $\Wasc_n$, the set
of weak ascent sequences of length $n$, that we then show (in Theorem
\ref{th:mainbiject}) to be a bijection.  This mapping is defined
recursively.  For $n=1$, we define $\LL(1)=(0)$.  Next let $n \ge 2$
and suppose that $\pi \in \RR_{n}$ is obtained by inserting $n$ into
active site labeled $i$ of $\tau \in \RR_{n-1}$.  The sequence
associated with $\pi$ is then $\LL(\pi)=(x_1, \dots, x_{n-1}, i)$,
where $(x_1, \dots,x_{n-1})=\LL(\tau)$.

\begin{example}
  The permutation $\pi = 6 2 7 5 4 1 3 8 $ corresponds to the sequence
  $x=(0,0,2,1,1,0,1,5)$.  It is obtained through the following
  recursive insertion of new maximal values into active sites.  The
  subscripts indicate the labels of the active sites.
  \begin{align*}
    _0 1 _1
    &\,\xrightarrow{x_2=0}\,  {_0} 2 {_1} 1 {_2} \\
    &\,\xrightarrow{x_3=2}\,  {_0} 2 {_1} 1 {_2} 3 {_3} \\
    &\,\xrightarrow{x_4=1}\,  {_0} 2 {_1} 4 \nyip  1 {_2} 3 {_3} \\
    &\,\xrightarrow{x_5=1}\,  {_0} 2 {_1} 5 {_2} 4 \nyip  1 {_3} 3 {_4} \\
    &\,\xrightarrow{x_6=0}\,  {_0} 6 \nyip 2 {_1} 5 {_2} 4 \nyip  1 {_3} 3 {_4} \\
    &\,\xrightarrow{x_7=1}\,  {_0} 6 \nyip 2 {_1} 7 {_2} 5 {_3} 4 \nyip  1 {_4} 3 {_5} \\
    &\,\xrightarrow{x_8=5}\,  \nyip 6 \nyip 2 \nyip 7 \nyip 5 \nyip 4 \nyip  1 \nyip 3 \nyip 8.
  \end{align*}
  In our terminology, we thus have
  $\LL(6, 2, 7, 5, 4, 1, 3, 8) = (0,0,2,1,1,0,1,5)$.
\end{example}

\begin{theorem}\label{th:mainbiject}
  \label{pwa}
  The map $\LL$ is a bijection from $\RR_{n}$ to $\Wasc_n$. 
	Furthermore, given $\pi \in \RR_n$ with $x=(x_1,\ldots,x_n) =\LL(\pi)$, 
	$\wasc(x) =\sss(\pi)$ and $x_n = \bbb(\pi)$ where $\sss$ is the number 
	of active sites in $\pi$ and $\bbb$ is the label of the site located just before the
  	largest entry of $\pi$.
\end{theorem}

\begin{proof}
  The integer sequence $\LL(\pi)$ encodes the construction of the
  $\wpattern$-avoiding permutation $\pi$ so the map $\LL$ is injective.
  In order to prove that $\LL$ is bijective we must show that the
  image $\LL(\RR_n)$ is the set $\Wasc_n$.  Let $\sss(\pi)$ be the
  number of active sites in the permutation $\pi$.  The recursive
  description of the map $\LL$ tells us that
  $x=(x_1, \dots,x_n)\in \LL(\RR_n)$ if and only if
  \begin{equation}\label{description}
    x'=(x_1, \dots, x_{n-1})\in \LL(\RR_{n-1})
    \quad\text{and}\quad 0\le x_n \le \sss\bigl(\LL^{-1}( x')\bigr)-1
  \end{equation}
  Recall that the leftmost active site is labeled 0 and the rightmost
  active site is labeled $\sss(\pi)-1$.  We will now prove by
  induction (on $n$) that for all $\pi \in \RR_n$, with associated
  sequence $\LL(\pi)=x=(x_1, \dots, x_n)$, one has
  \begin{equation}\label{properties}
    \sss(\pi)=2+\wasc(x)\quad\text{and}\quad \bbb(\pi)= x_n,
  \end{equation}
  where $\bbb(\pi)$ is the label of the site located just before the
  largest entry of $\pi$.  This will convert the above
  description~\eqref{description} of $ \LL(\RR_n)$ into the definition
  of weak ascent sequences, thereby concluding the proof.

  Let us focus on the properties~\eqref{properties}.  It is
  straightforward to see that they hold for $n=1$.  Next let us assume
  they hold for some $n-1$ where $n\ge 2$.  Let $\pi \in \RR_n$ be
  obtained by inserting $n$ into the active site labeled $i$ of
  $\tau \in \RR_{n-1}$.  Then $\LL(\pi)= x=(x_1, \dots, x_{n-1},i)$
  where $\LL(\tau)= x'=(x_1, \dots, x_{n-1})$.

  Every entry of the permutation $\pi$ that is smaller than $n$ is
  followed in $\pi$ by an active site if and only if it was followed
  in $\tau$ by an active site.  The leftmost site in $\pi$ also
  remains active.  The label of the active site preceding $n$ in $\pi$
  is $i=x_n$, and this proves the second property.  In order to
  determine $\sss(\pi)$, we must see whether the site following $n$ is
  active in $\pi$.  There are three cases to consider.  Before doing
  this recall that, by the induction hypothesis, we have
  $\sss(\tau)=2+\wasc(x')$ and $\bbb(\tau)=x_{n-1}$.

  {\it Case 1.} If $0\leq i< \bbb(\tau)=x_{n-1}$ then
  $\wasc(x)=\wasc(x')$ and the entry $n$ in $\pi$ is to the left of
  $n-1$ and there is at least one element in-between these.  This
  `in-between' element must be $<n-1$, so the site after $n$ in $\pi$
  cannot be active since it would lead to the creation of a
  $\wpattern$-pattern.  The number of active sites remains unchanged
  and $\sss(\pi) = \sss(\tau)=2+\wasc(x')=2+\wasc(x)$.

  {\it Case 2.} If $i=\bbb(\tau)=x_{n-1}$ then
  $\wasc(x)=1+\wasc(x')$ and the entry $n$ in $\pi$ is immediately to
  the left of $n-1$ in $\pi$.  Furthermore, there are no elements
  in-between $n$ and $n-1$.  The site that follows $n$ is therefore
  active, and $\sss(\pi ) = 1+\sss(\tau)=3+\wasc(x')=2+\wasc(x)$.

  {\it Case 3.} If $i>\bbb(\tau)=x_{n-1}$ then
  $\wasc(x)=1+\wasc(x')$ and the entry $n$ in $\pi$ is to the right of
  $n-1$.  The site that follows $n$ is therefore active, and
  $\sss(\pi ) = 1+\sss(\tau)=3+\wasc(x')=2+\wasc(x)$.
\end{proof}

\mynewpage
\section{A class of upper-triangular binary matrices}
\label{matricessection}
Dukes and Parviainen~\cite{DP} showed how the set of upper triangular
integer matrices whose entries sum to $n$ and which contain no zero
rows or columns are in one-to-one correspondence with ascent
sequences.  A property of that correspondence is that the number of
ascents in an ascent sequence equals the dimension of the
corresponding matrix less one, while the depth of the first non-zero entry in
the rightmost column corresponds to one plus the final entry of the
ascent sequence.

In this section we will present a similar construction for weak ascent
sequences.  This correspondence is different to
\cite{DP} in that the matrix entries are binary and rows of zeros will
be allowed.  The reason for this is that we would like the dimension
of a matrix to be one more than the number of weak ascents in the corresponding
weak ascent sequence. Further to this, and in keeping with the spirit
of \cite{DP}, we also wish to preserve the second property ``the depth
of the first non-zero entry in the rightmost column corresponds to the
final entry of the weak-ascent sequence''.

Let us first define the class of matrices we will be interested in and
then present the correspondence between these matrices and weak ascent
sequences. The notation $\dim(A)$ refers to the dimension of the
matrix $A$.

\begin{definition}\label{matrixdefn}
  Let $\newmatrices_n$ be the set of upper triangular square
  0/1-matrices $A$ that satisfy the following properties:
  \begin{enumerate}
  \item[(a)] There are $n$ 1s in $A$.
  \item[(b)] There is at least one 1 in every column of $A$.
  \item[(c)] For every pair of adjacent columns, the topmost 1 in the
    left column is weakly above the bottommost 1 in the right column.
\end{enumerate}
\end{definition}

All of the matrices in $\newmatrices_1,\ldots,\newmatrices_4$ are
shown in Table~\ref{allfourmatrices}.

\begin{table}[!ht]
$$
\def\arraystretch{1.3}
\begin{array}{c@{\quad}l} \hline
n & \newmatrices_n\\ \hline\\[-2.5ex]
1 & \Mat{1}\\[1ex]
2 & \Mat{1 & 1\\ 0 & 0}\!,\,
    \Mat{1 & 0\\ 0 & 1}\\[1.5ex]
3 & \Mat{1&1\\ 0&1}\!,\,
    \Mat{1&1&1\\ 0&0&0\\ 0&0&0}\!,\,
    \Mat{1&1&0\\ 0&0&1\\ 0&0&0}\!,\,
    \Mat{1&1&0\\ 0&0&0\\ 0&0&1}\!,\,
    \Mat{1&0&0\\ 0&1&1\\ 0&0&0}\!,\,
    \Mat{1&0&0\\ 0&1&0\\ 0&0&1}\\[2.6ex]
4 & \Mat{1&1&1\\ 0&0&1\\ 0&0&0}\!,\,
    \Mat{1&1&1\\ 0&0&0\\ 0&0&1}\!,\,
    \Mat{1&1&0\\ 0&0&1\\ 0&0&1}\!,\,
    \Mat{1&1&1\\ 0&1&0\\ 0&0&0}\!,\,
    \Mat{1&1&0\\ 0&1&1\\ 0&0&0}\!,\,
    \Mat{1&1&0\\ 0&1&0\\ 0&0&1}\!,\,
    \Mat{1&0&1\\ 0&1&1\\ 0&0&0}\!,\,
    \Mat{1&0&1\\ 0&1&0\\ 0&0&1}\!,\,
    \Mat{1&0&0\\ 0&1&1\\ 0&0&1}\!,\\[1.9ex]
  & \Mat{1&1&1&1\\ 0&0&0&0\\ 0&0&0&0\\ 0&0&0&0}\!,\,
    \Mat{1&1&1&0\\ 0&0&0&1\\ 0&0&0&0\\ 0&0&0&0}\!,\,
    \Mat{1&1&1&0\\ 0&0&0&0\\ 0&0&0&1\\ 0&0&0&0}\!,\,
    \Mat{1&1&1&0\\ 0&0&0&0\\ 0&0&0&0\\ 0&0&0&1}\!,\,
    \Mat{1&1&0&0\\ 0&0&1&1\\ 0&0&0&0\\ 0&0&0&0}\!,\,
    \Mat{1&1&0&0\\ 0&0&1&0\\ 0&0&0&1\\ 0&0&0&0}\!,\,
    \Mat{1&1&0&0\\ 0&0&1&0\\ 0&0&0&0\\ 0&0&0&1}\!,\\[2.4ex]
  & \Mat{1&1&0&0\\ 0&0&0&0\\ 0&0&1&1\\ 0&0&0&0}\!,\,
    \Mat{1&1&0&0\\ 0&0&0&0\\ 0&0&1&0\\ 0&0&0&1}\!,\,
    \Mat{1&0&0&0\\ 0&1&1&1\\ 0&0&0&0\\ 0&0&0&0}\!,\,
    \Mat{1&0&0&0\\ 0&1&1&0\\ 0&0&0&1\\ 0&0&0&0}\!,\,
    \Mat{1&0&0&0\\ 0&1&1&0\\ 0&0&0&0\\ 0&0&0&1}\!,\,
    \Mat{1&0&0&0\\ 0&1&0&0\\ 0&0&1&1\\ 0&0&0&0}\!,\,
    \Mat{1&0&0&0\\ 0&1&0&0\\ 0&0&1&0\\ 0&0&0&1}\\[2.4ex] \hline
\end{array}
$$
\caption{Matrices in our class of interest.\label{allfourmatrices}}
\end{table}

Our goal is to present and prove a bijection between the set of
  weak ascent sequences and the set of the matrices we have just
  introduced.  First, we give a recursive description of the bijection
  that allows us to prove some local properties of the correspondence
  in a simple way.  Following this we present a more direct definition
  of the bijection that involves a decomposition of the weak ascent
  sequence into decreasing subsequences.

Given a matrix $A\in\newmatrices$, let $\spone(A)$ be the index such
that $A_{\spone(A),\dim(A)}=1$ is the topmost 1 in the rightmost
column of $A$.  Such a value always exists since, by Definition~\ref{matrixdefn}(b), there
is at least one 1 in every column of $A$.  Let us define
$$
\reduce(A) = \bigl(B,\,\spone(A)-1\bigr)
$$
where $B$ is a copy of $A$ except that $B_{\spone(A),\dim(A)}=0$ (this
entry was 1 in $A$), and if this results in $B$ having a final column
of all zeros then we delete that column and the bottommost row so that
$B$ has dimension 1 less than $A$.

\begin{lemma}
  If  $A \in \newmatrices_n$, with $n\geq 2$, and
  $\reduce(A)=\bigl(B,\spone(A)-1\bigr)$, then $B\in\newmatrices_{n-1}$.
\end{lemma}

\begin{proof}
  Suppose $n$ and $A$ are as stated and
  $\reduce(A) = (B,\spone(A)-1)$.  Let us first observe that the
  number of 1s in $B$ is one less than the number of 1s in $A$, and is
  $n-1$. This shows property (a) of Definition~\ref{matrixdefn} is
  satisfied. Since $A\in\newmatrices_n$ there is at least one 1 in
  every column of $A$, let us consider what happens in the reduction
  from $A$ to $B$.  If there was a single 1 in the rightmost column of
  $A$, then it is removed along with that column and bottom row so
  that there is at least one 1 in every column of $B$.  Alternatively,
  if there was more than one 1 in the rightmost column of $A$, then
  changing the 1 at position $(\spone(A),\dim(A))$ to 0 will still
  ensure there is at least one more 1 in that column. This shows
  property (b) of Definition~\ref{matrixdefn} is satisfied.

  Showing that property (c) in Definition~\ref{matrixdefn} is
  preserved is a little bit more delicate.  Notice that in terms of
  our reduction we need only consider property (c) and how things
  change in terms of the rightmost two columns.  If there was only one
  1 in the rightmost column of A, then that column will not appear in
  $B$ so property (c) certainly holds true in this case.  If there is
  more than one 1 in the rightmost column of $A$, then removing it
  does not change the depth of the bottommost one in that column, so
  property (c) will still hold true.  In both cases, property (c)
  still holds.  Therefore $B \in \newmatrices_{n-1}$.
\end{proof}

Next we will define a matrix insertion operation that is complementary
to the removal operation $\reduce$.

\begin{definition}\label{defmatadd}
  Given a matrix $A \in \newmatrices_n$ and an integer
  $i\in [0,\dim(A)]$, let us define $\madd(A,i)$ as follows.
  \begin{enumerate}
  \item[(a)] If $i<\spone(A)-1$ then let
    $\madd(A,i)$ be the matrix $A$ with $A_{i+1,\dim(A)}$ changed from
    0 to 1.
  \item[(b)] If $i\geq \spone(A)-1$ then let
    $\madd(A,i)$ be the matrix $A$ with a new column of 0s added to
    the right and a new row of 0s appended to the bottom.  Then
    change the 0 at position $(i+1,\dim(A)+1)$ in $\madd(A,i)$ to 1.
  \end{enumerate}
\end{definition}

To illustrate Definition~\ref{defmatadd} let
$$
  A = \Mat{ \vrule width 0pt height 6pt
		1&1&0&1&0&0 \\
         0&0&1&1&1&0 \\
         0&0&1&0&0&0 \\
         0&0&0&0&0&1 \\
         0&0&0&0&0&0 \\
         0&0&0&0&0&1}\!.
$$
Then\\[-4.5ex]
$$\hspace{-3ex}
  \madd(A,2) =
   \Mat{	 \vrule width 0pt height 6pt
		1&1&0&1&0&0 \\
        0&0&1&1&1&0 \\
        0&0&1&0&0&\mathbf{1} \\
        0&0&0&0&0&1 \\
        0&0&0&0&0&0 \\
        0&0&0&0&0&1}\\
      \,\text{ and }\;
   \madd(A,4) =
   \Mat{ \vrule width 0pt height 8pt
		1&1&0&1&0&0&\mathbf{0} \\
        0&0&1&1&1&0&\mathbf{0} \\
        0&0&1&0&0&0&\mathbf{0} \\
        0&0&0&0&0&1&\mathbf{0} \\
        0&0&0&0&0&0&\mathbf{1} \\
        0&0&0&0&0&1&\mathbf{0} \\
        \mathbf{0}&\mathbf{0}&\mathbf{0}&\mathbf{0}&\mathbf{0}&\mathbf{0}&\mathbf{0}}
		\vrule width 0pt height 6pt
      \!.
$$

\begin{lemma}\label{lemma2}
  Let $n \geq 2$ and $B \in \newmatrices_{n-1}$. Let
  $i \in [0,\dim(A)]$ and define $A=\madd(B,i)$.  Then
  $A \in \newmatrices_{n}$ and $\spone(A)=i+1$.
\end{lemma}

\begin{proof}
  Let $n$, $i$, and $B$ be as stated in the lemma.  In
  Definition~\ref{defmatadd}, the two operations (a) and (b) increase
  the number of 1s in the matrix by 1, so the number of 1s in the
  matrix $A=\madd(B,i)$ will be $n$.  This means matrix $A$ satisfies
  Definition~\ref{matrixdefn}(a).  Similarly, if addition rule (a) is
  used then the number of 1s in a column of $A$ is at least as many as
  in $B$, so $A$ satisfies Definition~\ref{matrixdefn}(b).  If
  addition rule (b) is used then the new column that appears has
  precisely one 1, so again $A$ satisfies
  Definition~\ref{matrixdefn}(b).  Let us now consider the cases
  $i<\spone(B)-1$ and $i\ge \spone(B)-1$ separately.

  If $i<\spone(B)-1$ then matrix $A$
  is created by inserting a 1 into position $(i+1,\dim(B))$ of $B$
  which is above the topmost 1 in that column.  Consequently in the
  new matrix $A$ one has $\spone(A) = i+1$.  Furthermore, the
  positions of the topmost 1 in the second to last column and the
  bottommost 1 in the final column remain unchanged, so $A$
  satisfies Definition~\ref{matrixdefn}(c).

  If $i\ge \spone(B)-1$ then $A$ is created from $B$ by adding a new
  column and row, and inserting a 1 at position
  $(i+1,\dim(B)+1)$. Notice that the topmost 1 in column $\dim(B)$ is
  at position $(\spone(B),\dim(B))$. The bottommost 1 in the final
  column of $A$ is now at position $(i+1,\dim(B)+1)$.  Since
  $\spone(B) -1 \leq i$ we have $\spone(B) \leq i+1$ and again $A$
  satisfies Definition~\ref{matrixdefn}(c).
\end{proof}

\begin{lemma}
  Let $B \in \newmatrices_n$ and let $i \in [0,\dim(B)]$. Then
  $$\reduce\bigl(\madd(B,i)\bigr) = (B,i)
  $$
  and, if $n\geq 2$,
  $$
  \madd\bigl(\reduce(B)\bigr) = B.
  $$
\end{lemma}

\begin{proof}
  Let $A=\madd(B,i)$. From Lemma~\ref{lemma2} we have $\spone(A)=i+1$
  and the removal operation when applied to $A$ will yield $(C,i)$ for
  some matrix $C$. We need to show $B=C$ for the two different cases
  of Definition~\ref{defmatadd}.  Suppose that $i<\spone(B)-1$. Then
  $A$ is a copy of $B$ with a new 1 at position $(i+1,\dim(B))$, which
  becomes the topmost one in that column.  The reduction operation
  applied to $A$ removes that topmost one in the rightmost column and
  the resulting matrix is $C=B$. A similar argument holds for the case
  $i \ge \spone(B)-1$. This establishes the first part of our lemma.

  Suppose $B \in \newmatrices_n$ is a matrix that has only one
  $1$-entry in the last column.  Then $\reduce(B)= (C, i)$, where $C$ is
  the matrix that we obtain by deleting the rightmost column and
  bottommost row of $B$ and $i = \spone(B)-1$.  By the property (c) of
  Definition~\ref{matrixdefn} $\spone(C)\leq \spone(B)$. So,
  $i\geq \spone(C)-1$ and by Definition~\ref{defmatadd}(b) we have
  that the matrix $A=\madd(C,i)$ is the matrix that we obtain by
  appending a column with a single $1$ in row $i+1$ to the right and
  an all-zeros row to the bottom of $C$. Hence $A=B$.  Suppose now
  that $B \in \newmatrices_n$ is a matrix with more than one $1$ in
  the rightmost column.  Then $\reduce(B)= (C, i)$, where $C$ is the
  matrix that we obtain by exchanging the topmost $1$ in the rightmost
  column to $0$ and $i = \spone(B)-1$.  Note that due to this we have
  $\spone(C)>\spone(B)$ and $i<\spone(C)-1$.  By
  Definition~\ref{defmatadd}(a) $A=\madd(C,i)$ is the matrix that we
  obtain by changing the $0$ in the $(i+1)$th row in the rightmost
  column of $C$ to $1$, hence we have $A=B$.
\end{proof}

Let us now define a mapping $\MM$ from $\newmatrices_n$ to integer
sequences of length $n$.

\begin{definition}
  For $n=1$ let $\MM(\Mat{1}) = (0)$.  Now let $n\geq 2$ and suppose
  that the removal operation, when applied to $A\in \newmatrices_n$
  gives $\reduce(A)=(B,i)$.  Then the sequence associated with $A$ is
  $\MM(A) = (x_1,\ldots,x_{n-1},i)$ where
  $(x_1,\ldots,x_{n-1}) = \MM(B)$.
\end{definition}

\begin{theorem}
  The mapping $\MM: \newmatrices_n\to\Wasc_n$ is a bijection.
\end{theorem}

\begin{proof}
  Since the sequence $\MM(A)$ encodes the construction of the matrix
  $A$, the mapping $\MM$ is injective.  We have to prove that the
  image of $\newmatrices_n$ is the set $\Wasc_n$.  By definition,
  $x=(x_1,\ldots,x_n) \in \MM(\newmatrices_n)$ if and only if
  \begin{align}
    x'=(x_1,\ldots,x_{n-1}) \in \MM(\newmatrices_{n-1}) \quad\mbox{and}\quad
    x_n \in [0,\dim(\MM^{-1}(x'))]. \label{condone}
  \end{align}
  We will prove by induction on $n$ that for all
  $A \in \newmatrices_n$, with associated sequence
  $\MM(A) = x = (x_1,\ldots,x_n)$, one has
  \begin{align}
    \dim(A) = \wasc(x)+1 \quad\mbox{and}\quad \spone(A) = x_n+1. \label{condtwo}
  \end{align}
  This will convert the description \eqref{condone} above into the
  definition of weak ascent sequences, thus concluding the proof.

  Let us examine the two statements of \eqref{condtwo} more
  closely. They hold for $n=1$. Assume they hold for $n-1$ with
  $n\geq 2$ and let $A=\madd(B,i)$ for $B \in \newmatrices_{n-1}$.  If
  $\MM(B) = (x_1,\ldots,x_{n-1})$ then
  $\MM(A) = (x_1,\ldots,x_{n-1},i)$.  Lemma~\ref{lemma2} gives us that
  $\spone(A)=i+1$ and it follows that
  $$\dim(A) =
  \begin{cases}
    \dim(B)=\wasc(x')+1 = \wasc(x)+1 & \mbox{ if }i<x_{n-1}\\
    \dim(B)+1 = \wasc(x')+1+1=\wasc(x)+1 & \mbox{ if }i\ge x_{n-1}.
  \end{cases}
  $$
  The result follows from this.
\end{proof}

\begin{example}\label{starone} Let us construct the matrix $A$ that
  corresponds to the weak ascent sequence
  $x=(0,0,2,1,1,0,1,5) \in \Wasc_8$; that is, $\MM(A)=x$.  To begin, we
  have $\MM(\Mat{1})=(0)$. From this,
  $$
  \begin{array}{cccccccccl}
    \Mat{1} & \xrightarrow{x_2=0} & \Mat{1&1 \\ 0&0}
    & \xrightarrow{x_3=2} & \Mat{1&1&0 \\ 0&0&0 \\ 0&0&1}
    & \xrightarrow{x_4=1} & \Mat{1&1&0 \\ 0&0&1 \\ 0&0&1}
    & \xrightarrow{x_5=1} & \Mat{1&1&0&0 \\ 0&0&1&1 \\ 0&0&1&0 \\ 0&0&0&0} &\\[1em]
    &&&&
    & \xrightarrow{x_6=0} & \Mat{1&1&0&1 \\ 0&0&1&1 \\ 0&0&1&0 \\ 0&0&0&0}
    & \xrightarrow{x_7=1} & \Mat{1&1&0&1&0 \\ 0&0&1&1&1 \\ 0&0&1&0&0 \\ 0&0&0&0&0 \\ 0&0&0&0&0 }& \\[1.5em]
    &&&&&&& \xrightarrow{x_8=5} & \Mat{     \vrule width 0pt height 6pt  1&1&0&1&0&0 \\ 0&0&1&1&1&0 \\ 0&0&1&0&0&0 \\ 0&0&0&0&0&0 \\ 0&0&0&0&0&0 \\ 0&0&0&0&0&1      \vrule width 0pt height 6pt }& \!\!=\; A
  \end{array}
  $$
\end{example}

We can offer a more direct description of the bijection in terms
  of the decreasing subsequence decomposition of a weak ascent
  sequence.  A weak ascent sequence can be seen as a sequence of
  decreasing subsequences or runs.
  For instance, $w=00211015$ can be decomposed into
  decreasing runs as $w=0|0|21|10|1|5$.  Construct the matrix $M$ from
  $w$ as follows: column $i$ of $M$ is formed from the $i$th
  decreasing run of $w$ whereby there is a $1$-entry at position $(i,j)$
  if the $i$th decreasing run contains the value $j-1$.

Let us also mention that the inverse mapping to this consists of
  labelling all the $1$-entries of a matrix $M \in \newmatrices_n$ with
  labels 1 to $n$ (in the manner specified in
  Definition~\ref{posdef}).  Then the $i$th element of the weak
  ascent sequence is equal to $j$ if and only if the entry labelled $i$ is in the
  $(j+1)$th row of the matrix $M$.

It is straightforward to see how several simple statistics get translated through the bijection $\MM$.
\begin{proposition}
Let $w=(w_1,\ldots,w_n) \in \Wasc_n$ and suppose that $M \in \newmatrices_n$ is such that $\MM(M)=w$.
Then
\begin{itemize}
\item the number of occurrences of $j$ in $w$ equals
    the sum of the entries of the $(j+1)$th row in $M$,
\item the number of weak ascents in $w$ equals the dimension of $M$
  reduced by $1$,
\item the length of the final decreasing run is the
    sum of the entries in the rightmost column of $M$,
\item $w_n$, the last entry of $w$, is equal to $\spone(M)-1$.
\end{itemize}
\end{proposition}

Next, we will characterize those matrices that correspond, via our bijection $\MM$, to ascent sequences.

Given a word $w=w_1\dots w_n$ let an entry $w_i$ with
  $w_{i-1}=w_i$ be called a \emph{plateau}.
Each weak ascent ``creates'' a new column.  An ascent
  corresponds to a bottommost $1$-entry in a new column which is
  strictly south-east of the topmost $1$ in the previous column.  On
  the other hand a plateau corresponds to a bottommost $1$-entry in a
  new column which is in the same row as the topmost $1$ in the
  previous column.

We introduce the
  terminology \emph{weak} and \emph{strong} entries for these two
  special kind of $1$s in the matrix.
  \begin{definition}Let a $1$-entry in a $0/1$-matrix be called
    \emph{weak} entry if (a) there are only $0$s below the $1$-entry
    in its column, and (b) there is a $1$ to the left of it in the
    same row such that there are only $0$s above this $1$. On the
    other hand, let a $1$-entry in a $0/1$-matrix be called
    \emph{strong} entry if (a) there are only $0$s below the $1$-entry
    in its column, and (b) the top most $1$ to the column to the left
    is in a smaller indexed row.
\end{definition}
  \begin{example}
    In matrix $C$ the entry at position
    $(2,4)$ is the only weak entry, $C_{2,2}$, $C_{3,3}$, $C_{2,4}$
    and $C_{5,5}$ are strong $1$-entries. In the matrix $D$ the
    entries at positions $(1,2)$ and $(2,5)$ are the weak entries
    and $D_{2,3}$, $D_{3,4}$ and $D_{3,6}$ are the strong $1$-entries.
    $$
	C= \Mat{     \vrule width 0pt height 6pt
		1&1&{\bf{0}}&1&0&0 \\
		0&1&{\bf{1}}&{\bf{1}}& 1&0 \\
		0&0&1&{\bf{0}}&0&0 \\
		0&0&0&{\bf{0}}&0&0 \\
		0&0&0&{\bf{0}}&0&0 \\
		0&0&0&{\bf{0}}&0&1      \vrule width 0pt height 6pt }
	\quad
	\quad D  = \Mat{     \vrule width 0pt height 6pt
		{\bf{1}}&{\bf{1}}&0&{\bf{0}}&1&1 \\
		0&{\bf{0}}&1&{\bf{1}}&{\bf{1}}&0 \\
		0&{\bf{0}}&0&1&{\bf{0}}&1 \\
		0&{\bf{0}}&0&0&{\bf{0}}&0 \\
		0&{\bf{0}}&0&0&{\bf{0}}&0 \\
		0&{\bf{0}}&0&0&{\bf{0}}&0     \vrule width 0pt height 6pt}
    $$
  \end{example}

The main observation is that in an ascent sequence the
  appearance of a plateau $w_i$ restricts the possible values of $w_j$
  for $j\geq i+1$, though it does not influence the values in a weak
  ascent sequence.  Hence, if $w$ is an ascent
  sequence then in the corresponding matrix a $1$ can only appear in
  row $i$ in the $j$th column if there are in the previous $j-1$
  columns at least $i-2$ strong entries (column 1 cannot by definition contain a strong entry).

Thus, we have the following corollary of our bijection.
  \begin{corollary} A matrix $A\in\newmatrices$ corresponds to an
    ascent sequence via the bijection $\MM$ if for every $1$-entry
    $A_{i,j}$ there are at least $i-2$ strong entries in columns
    $1,2,\ldots, j-1$.
  \end{corollary}

\mynewpage
\section{A class of factorial posets}
\label{posetssection}

In this section we will define a mapping from the set of matrices
$\newmatrices$ to a set of labeled posets and prove that this mapping
is a bijection.  First let us recall the definition of a factorial
poset from ~\cite{factorialposets}.  To begin, a poset $P$ on the
elements $\{1,\ldots,n\}$ is \emph{naturally labeled} if
$i <_P j$ implies $i<j$.  The poset whose Hasse diagram is depicted in
Figure~\ref{figpos} is a naturally labeled poset.

\begin{definition}[\cite{factorialposets}]
  A naturally labeled poset $P$ on $[1,n]$ such that,
  for all $i,j,k \in [1,n]$, we have
  $$i<j \mbox{ and } j <_P k \;\implies\; i <_P k
  $$
  is called a \emph{factorial poset}.
\end{definition}
Factorial posets are (2+2)-free. This fact, and further properties of
factorial posets can be found in \cite{factorialposets}.

\begin{definition}[The mapping $\Psi$]\label{posdef}
  Let $A \in \newmatrices_n$. Form a matrix $B$ as follows. Make a
  copy of $A$. Beginning with the leftmost column, and within each
  column one goes from bottom to top, replace every 1 that appears
  with the elements $1,2,\ldots,n$. Further, define a partial order
  $(P,<)$ on $[1,n]$ as follows: $i <_P j$ if the index of the column
  that contains $i$ is strictly less than the index of the row that
  contains $j$.  Let
  $$P=\PosetMap(A)
  $$ be the resulting poset.
\end{definition}

Diagrammatically the relation in Definition~\ref{posdef} is equivalent
to $i$ being north-west of $j$ in the matrix and the ``lower hook'' of
$i$ and $j$ being strictly beneath the diagonal:\smallskip

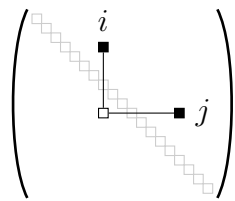
\myonematrix

Note that the set of entries contained in the first $s$ columns for an $s$ is
the complete set $\{1,2,\ldots, {s_{k}}\}$ for some $s_k$.

\begin{example}
  Consider the matrix $A$ from Example~\ref{starone}. Form matrix $B$
  by relabeling the 1s in the matrix according to the rule.
  $$
  \begin{array}{ccc}
    A =
    \begin{bmatrix}
      1&1&0&1&0&0 \\
      0&0&1&1&1&0 \\
      0&0&1&0&0&0 \\
      0&0&0&0&0&0 \\
      0&0&0&0&0&0 \\
      0&0&0&0&0&1
    \end{bmatrix}
       &\mapsto &
    B =
    \begin{bmatrix}
      1&2&0&6&0&0 \\
      0&0&4&5&7&0 \\
      0&0&3&0&0&0 \\
      0&0&0&0&0&0 \\
      0&0&0&0&0&0 \\
      0&0&0&0&0&8
    \end{bmatrix}
  \end{array}
  $$
  This gives the poset $(P,<)= \PosetMap(A)$ with the
  following relations:
  \begin{itemize}
  \item $1<_P 3,4,5,7,8$
  \item $2 <_P 3,8$
  \item $3,4,5,6,7 <_P 8$
  \end{itemize}
  The Hasse diagram of this poset is illustrated in
  Figure~\ref{figpos}.
  \begin{figure}
  $$
  \begin{tikzpicture}[xscale=0.7, yscale=0.9]
    \tikzstyle{every node} = [font=\footnotesize];
    \tikzstyle{disc} = [
      circle,fill=black,draw=black,
      minimum size=4pt, inner sep=0pt];
    \path
    node [disc] (1) at (1, 1) {}
    node [disc] (2) at (3, 1) {}
    node [disc] (6) at (7, 1) {}
    node [disc] (3) at (0, 3) {}
    node [disc] (4) at (2, 3) {}
    node [disc] (5) at (4, 3) {}
    node [disc] (7) at (6, 3) {}
    node [disc] (8) at (6.5, 5) {};
    \draw 
    (1) node[left=2pt] {$1$} -- (3) node[left=2pt]  {$3$}
    (1) -- (4) node[left=2pt]  {$4$}
    (1) -- (5) node[above left]  {$5$}
    (1) -- (7) node[above left]  {$7$}
    (6) node[right=2pt] {$6$} -- (8) node[right=2pt]  {$8$}
    (2) node[right=2pt] {$2$} -- (3) {}
    (3) -- (8)
    (4) -- (8)
    (5) -- (8)
    (7) -- (8)
    ;
  \end{tikzpicture}
  $$
  \caption{A weakly $(3+1)$-free factorial poset.\label{figpos}}
  \end{figure}
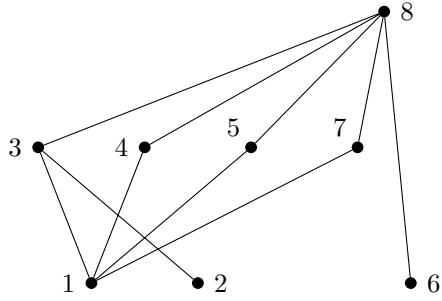
\end{example}

The mapping $\PosetMap$ is a mapping from $\newmatrices_n$ to a set of
labeled (2+2)-free posets on the set $[1,n]$, which we will now
define.

Let $P$ be a factorial poset on $[1,n]$.  We say that $P$ contains a
\emph{special 3+1} if there exist four distinct elements $i<j<j+1<k$
such that the poset $P$ restricted to $\{i,j,j+1,k\}$ induces the 3+1
poset with $i<_P j <_P k$:
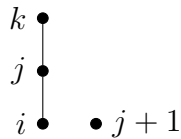
$$
\begin{tikzpicture}[scale=0.7]
  \tikzstyle{disc} = [
  circle,fill=black,draw=black,
  minimum size=4pt, inner sep=0pt];
  \path
  node [disc] (i) at (0, 1) {}
  node [disc] (j) at (0, 2) {}
  node [disc] (k) at (0, 3) {}
  node [disc] (j+1) at (1, 1) {};
  \draw
  (i) node[left=2pt] {$i$} -- (j) node[left=2pt]  {$j$}
  (j) -- (k) node[left=2pt] {$k$}
  (j+1) node[right=2pt] {$j+1$};
\end{tikzpicture}
$$
If $P$ does not contain a special 3+1 we
say that $P$ is \emph{weakly $(3+1)$-free}.
Let $\WPoset_n$ be the set of weakly $(3+1)$-free factorial posets
on $[1,n]$.

\begin{theorem}
  Let $\PosetMap$ be as in Definition~\ref{posdef}. If
  $A \in \newmatrices_n$, then the poset $P=\PosetMap(A)$ is factorial
  and weakly $(3+1)$-free. That is $P\in\WPoset_n$ so that
  $$\PosetMap:\newmatrices_n \to \WPoset_n.
  $$
\end{theorem}

\begin{proof}
  Let $A \in \newmatrices_n$ and $P=\PosetMap(A)$.  Given
  $i \in [1,n]$, we define the strict downset
  $D(i) = \{ j \in [1,n]: j <_P i\}$ of $i$.  A defining
  characteristic of a (2+2)-free poset is that the collection
  $\{D(i):i\in[1,n]\}$ of strict downsets can be linearly ordered by
  inclusion. Similarly, a defining characteristic of a factorial poset
  is that each strict downset is of the form $[1,k]$ for some $k<n$.
  From Definition~\ref{posdef}, it is clear why this must be the case
  for $P$: In constructing $P$ from the matrix $A$, the intermediate
  matrix $B$ contains all entries in $[1,n]$ exactly once.  If $j$
  appears in row $t$ of $B$, then the strict downset $D(j)$ will
  consist of all $i$'s that appear in columns 1 through to $t-1$
  (inclusive).  All elements that appear in row $t$ of $B$ have the
  same strict downset.  Furthermore, since the entries $1$, $2$,
  \dots, $n$ in $B$ are such that their indices appear from left to
  right in increasing order, the strict downset of every element must
  be $[1,k]$ for some $k<n$. Thus $P$ is factorial.

  Let us now suppose that $P$ contains an induced subposet on the four
  elements $i<j<j+1<k$ that forms a special 3+1.  In particular,
  $i <_P j <_P k$.  Consider the matrix entries in $B$ that correspond
  to $i$, $j$, $j+1$ and $k$. Suppose that $i$ is at position
  $(i_1,i_2)$ in $B$, and that $j$ and $k$ are at positions
  $(j_1,j_2)$ and $(k_1,k_2)$, respectively.  The hooks formed from
  $(i,j)$ and $(j,k)$ must be beneath the diagonal, so we must have
  $i_1\leq i_2< j_1 \leq j_2 < k_1 \leq k_2$.  Consider now
  $\ell = j+1$ that is at position $(\ell_1,\ell_2)$ in $B$.  This
  element must appear either (a) in the same column as $j$ and
  strictly above it, or (b) in the next column and in a row weakly
  below $j$.  For case (a) this means $\ell_1 \leq \ell_2 = j_2$, from
  which we find that $\ell <_P k$, but this cannot happen since $\ell$
  and $k$ are incomparable.  For case (b) this means
  $j_2 \leq \ell_1 \leq \ell_2$, from which we find that $i<_P\ell$,
  but this cannot happen since $\ell$ and $i$ are incomparable.
  Therefore $P$ cannot contain a special 3+1. In other words, $P$ is
  weakly (3+1)-free and $\PosetMap:\newmatrices_n \to \WPoset_n$.
\end{proof}

We next define a function $\ptm$ that maps posets in $\WPoset_n$ to
matrices. We shall show that $\ptm$ is the inverse of $\PosetMap$.

\begin{definition}\label{bfromc}
  Given $P \in \WPoset_n$, suppose there are $k$ different strict
  downsets of the elements of $P$, and that these are
  $D_0 = \emptyset$, $D_1$, $\ldots$, $D_{k-1}$.  By convention we
  also let $D_k = [1,n]$.  Suppose that $L_i$ is the set of elements
  $p \in P$ such that $D(p) =
  D_i$; these are called \emph{level sets}.
  Let $C$ be the matrix with
  $C_{i,j} = L_{i-1} \cap (D_{j}\backslash D_{j-1})$ for all
  $i,j \in [1,k]$, see \cite{compositionmatrices} for details.
  Start with $B$ as a copy of $C$ and then repeat the following steps
  until there is no $i$ that satisfies the condition of A1:
  \begin{enumerate}
  \item[A1.] Choose the first column, $i$ say, of $B$ that contains either a set of size $>1$ or two elements such that the element in the higher row is less is value to the element in the lower row.
  \item[A2.] With respect to the usual order on $\N$, let $\ell$ be
    the smallest entry in column $i$.
  \item[A3.] Introduce a new empty column between columns $i-1$ and
    $i$ so that the old column $i$ is now column $i+1$.
  \item[A4.] Move $\ell$ one column to the left (to the current column
    $i$) and set $j=1$.
  \item[A5.] If $\ell+j$ is strictly above $\ell+j-1$, then move it one
    column to the left, increase $j$ by 1, and repeat A5. Otherwise,
    go to A6. (The outcome of step A5 will be that the non-empty
    singleton sets in column $i$, from bottom to top, are
    $\ell, \ell+1,\ldots, \ell+t$ for some $t \geq 0$.)
  \item[A6.] Introduce a new row of empty sets between rows $i$ and
    $i+1$ of $B$. Matrix $B$ will have increased in dimension by 1.
  \end{enumerate}

  Finally, let $\ptm(P)$ be the result of replacing the singletons in
  $B$ with ones and the empty sets in $B$ with zeros.
\end{definition}

\begin{example}
  Let $P$ be the poset in Figure~\ref{figpos}.  The strict downsets of
  $P$ are
  $$D_0=\emptyset, D_1=\{1\}, D_2=\{1,2\}\text{ and }
  D_3 = \{1,2,3,4,5,6,7\}.
  $$ The level sets of $P$ are
  $$L_0=\{1,2,6\}, L_1=\{4,5,7\}, L_2=\{3\}\text{ and }
  L_3=\{8\}.
  $$ This gives the matrix
  $$C=
    \begin{bmatrix}
      \{1\} & \{2\} & \{6\} & \emptyset \\
      & \emptyset & \{4,5,7\} & \emptyset \\
      & & \{3\} & \emptyset \\
      & & & \{8\}
    \end{bmatrix}.
  $$
  Before continuing we would like to point out that the matrix
  obtained from $C$ by replacing each set $C_{i,j}$ by its cardinality
  $|C_{i,j}|$ corresponds, via \cite[\S 3.1]{compositionmatrices},
  to the unlabeled version of the poset
  which is (2+2)-free.  Continuing, by applying A1 we find that $i=3$
  is the first column satisfying A1.  The smallest
  number in this column is $3$. Furthermore, $4$ is above $3$ so we move
  $4$ to the 3rd column, but $5$ is in the same row as $4$. So the
  largest contiguous sequence starting from $3$ as we go up from it in
  that column (and skip rows if we wish) is $\{3,4\}$. We next go to step
  A6 and insert a new row with empty sets below the row of $\{3\}$. Hence,
  the outcome of A6 will be
  $$B=
  \begin{bmatrix}
    \{1\} & \{2\} & \emptyset & \{6\} & \emptyset \\
    & \emptyset & \{4\} & \{5,7\} & \emptyset \\
    & & \{3\} & \emptyset &\emptyset \\
    & &  & \emptyset &\emptyset \\
    & & & & \{8\}
  \end{bmatrix}.
  $$
  As column $i=4$ satisfies the statement of A1, we go through steps A2-A6 and find the outcome of A6 to be
  $$B=
  \begin{bmatrix}
    \{1\} & \{2\} & \emptyset & \{6\} & \emptyset & \emptyset \\
    & \emptyset & \{4\} & \{5\} & \{7\} & \emptyset \\
    & & \{3\} & \emptyset & \emptyset &\emptyset \\
    & &  & \emptyset &\emptyset & \emptyset  \\
    & &  & &\emptyset & \emptyset  \\
    & & & &  & \{8\}
  \end{bmatrix}.
  $$
  Since there are now no columns that satisfy A1, we
  replace singletons with 1s and emptysets with 0s to find
  $$\ptm(P) =
  \begin{bmatrix}
    1&1&0&1&0&0 \\
    0&0&1&1&1&0 \\
    0&0&1&0&0&0 \\
    0&0&0&0&0&0 \\
    0&0&0&0&0&0 \\
    0&0&0&0&0&1
  \end{bmatrix}.
  $$
\end{example}

\begin{theorem} \label{welldefinedptm}
  Let $\ptm$ be the mapping from Definition~\ref{bfromc}. If $P$ is a
  poset in $\WPoset_n$, then the matrix $\ptm(P)$ is in
  $\newmatrices_n$ so that
  $$\ptm: \WPoset_n \to \newmatrices_n.
  $$
\end{theorem}

\begin{proof}
  Let $P \in\WPoset_n$.  Since $P$ is a factorial poset, we know that
  the strict downsets of elements all have the same contiguous form
  $[1,\ell]$ for some $\ell\in [0,n-1]$.  Let us suppose that there
  are $k$ different strict downsets of the elements of $P$, and that
  these are $D_0 = \emptyset$, $D_1$, $\ldots$, $D_{k-1}$.  Let us
  further suppose that $L_i$ is the set of elements $p \in P$ such
  that $D(p) = D_i$, the elements at {\it{level $i$}} of the poset.
  Next let $C$ be the matrix with
  $$C_{i,j} = L_{i-1} \cap (D_{j}\backslash D_{j-1}) \text{ for all
  }i,j \in [1,k].$$

  The matrix $C$ is upper triangular and is in fact a partition
  matrix.  A partition matrix is an upper triangular matrix whose
  entries form a set partition of an underlying set with the
  additional property: for all $1\leq a<b\leq n$, the column
  containing $b$ cannot be left of that containing $a$. See
  \cite{partition_and_composition_matrices} for further details on
  partition matrices.

  Moreover, since $P$ is weakly (3+1)-free, the structure of the
  matrix $C$ is further restricted in the following sense:
  \begin{description}
  \item[Property 1]
    The matrix $C$ does not contain four entries $i$,
    $j$, $j+1$ and $k$ such that the hooks for the pairs $(i,j)$ and
    $(j,k)$ are both below the main diagonal, whereas the hooks, if defined
    between $j+1$ and each of $i,j,k$ are on or above the main
    diagonal.
  \end{description}
  (Note that the hook for a pair $(i,j)$ is only defined when the
    entries are in strict north-west relative position.  When the hook
    for two entries is not defined, then the corresponding entries in
    the poset are incomparable.)

  Next, let us consider $B$ that is constructed from $C$ in
  Definition~\ref{bfromc}.  When a column of $B$ is split into two (as
  per A3), a new empty row is added in step A6 which preserves the
  upper-triangular property.  Also, there can be no empty columns in
  $B$ since the dissection step A4 ensures a set of size at least 2 is
  split into a singleton set (that will appear in the new left column)
  and the set difference (that will appear one place to its right).
  Since $C$ is upper triangular, the construction of $B$ ensures it is
  upper triangular.

  Furthermore, by construction, it can never be the case that on
  completion of all A1--A6, the entry $a+1$ appears above $a$ in the
  column to its right.  If it were, then rule A5 would not have
  been executed properly.  So the matrix $B$ is such that the highest
  indexed entry in every column is the highest entry in that column,
  $a$ say, is weakly above the smallest entry (with respect to the
  order on $\N$) in the subsequent column $a+1$ (which is also the
  lowest in that column). This condition absorbs Property~1 
  when one
  considers the final pair of columns and the entries $j$ and $j+1$.

  The replacement of all singleton sets with ones and empty sets with
  zeros results in a matrix $\ptm(P)$ with the following properties:
  \begin{itemize}
  \item $\ptm(P)$ contains $n$ ones and is upper triangular.
  \item There are no columns consisting of all-zeros, but there can be
    rows of all-zeros.
  \item The topmost one in every column is weakly above the bottommost
    1 in the column to its right.
  \end{itemize}
  Therefore, we have $\ptm(P) \in \newmatrices_n$.
\end{proof}

\begin{theorem}
  The mapping $\PosetMap:\newmatrices_n \to \WPoset_n$ is a bijection.
\end{theorem}

\begin{proof}
  We start by showing that $\PosetMap$ is injective. Suppose that $A$
  and $A'$ are two different matrices in $\newmatrices_n$. As there
  are $n$ 1s in each of the matrices $A$ and $A'$, and they are
  different, there must be at least two positions in which they
  differ.
  Consider the intermediate matrices $B$ and $B'$ in the
    construction and let $a$ be the smallest label that appears in a
    different position in $B$ and in $B'$.  Fix $i$, $i'$, $j$, $j'$ so
    that $B_{i,j}=a=B'_{i',j'}$.  If $i\neq i'$ then the strict
    downset of $a$ in $\PosetMap(A)$ is different to its strict
    downset in $\PosetMap(A')$.  If $i=i'$ then $j\neq j'$. Let us
    suppose that $j<j'$. We necessarily have $j'=j+1$ as otherwise the
    column $j+1$ would be empty in $B'$ and $A'$.  Consequently the
    label $a-1$ is in column $j$ in both $B$ and $B'$ (otherwise
    column $j$ would be empty in $B'$ and $A'$), and therefore $a-1$
    is in a row below row $i$.  However this implies that, in $B'$,
    the topmost entry in column $j$ (equal to $a-1$) is below the
    bottommost entry in column $j+1$ (equal to $a$), contradicting
    $A'\in \newmatrices_n$.  Therefore
    $\PosetMap(A) \neq \PosetMap(A')$.

To prove that $\PosetMap$ is surjective, we will show that
$\PosetMap(\ptm(P))=P$, thereby establishing $\ptm$ as the inverse of
$\PosetMap$.  Let $P \in \WPoset_n$ that has $k$ different levels
$L_0,\ldots,L_{k-1}$ and down sets $D_0,\ldots,D_{k-1}$.  Let
$M=\ptm(P)$ be the matrix that satisfies Definition~\ref{matrixdefn}.
Suppose that $Q=\PosetMap(\ptm(P)) = \PosetMap(M)$.

As $M \in \newmatrices_n$ we know, by Theorem~\ref{welldefinedptm},
that $Q \in \WPoset_n$.  The poset $Q=\PosetMap(M)$ that we construct
using Definition~\ref{posdef} is such that the level $j$ of the poset
$Q$ corresponds to the set of elements in the $(j+1)$th non-zero row of
$M$ once the labelleing of the 1s in $M$ that is described in
Definition~\ref{posdef} has taken place. Given that the $(j+1)$th
non-zero row of $M$ corresponded to the $j$th level of $P$, we have
that $L_{j+1}(Q) = L_{j+1}(P)$ for all $j$.  Let $i_1,\ldots,i_k$
denote the indices of the $k$ non-empty rows of $M$.  The elements of
$D_{j}(Q)$ are all of those matrix entries (with the labelling of
Definition~\ref{posdef}) weakly to the right of column $i_j$.  As $M$
was constructed from $P$ using a process of separating out columns
while creating empty rows, the $j$th downset of $D_j(P)$ will coincide
with that of $D_j(Q)$.
\end{proof}

Just as we did in the previous section, we can see how statistics
between these two sets are translated:

\begin{proposition}
Suppose that $M \in \newmatrices_n$ and $P=\PosetMap(M)$.
Then
\begin{itemize}
\item the sum of the top row of $M$ is the number of minimal elements in $P$,
\item the number of non-zero rows in $M$ equals the number of levels in the poset $P$.
\end{itemize}
\end{proposition}

\mynewpage
\section{Pattern-avoiding inversion sequences and enumeration}
\label{inversionssection}

The study of patterns in inversion sequences was recently considered by
Corteel et al.~\cite{corteel} and continued
throughout several papers \cite{Auliphd, AuliSergi1, AuliSergi2,
  elizalde, Lin, Mansour}. In a recent paper
Auli and Elizalde~\cite{elizalde} focused on vincular patterns in
inversion sequences. We recall some important definitions.

It is well known that a permutation of $[1,n]$ can be encoded by an integer
sequence $e_1e_2\ldots e_n$, where
$e_i$ is the number of
larger elements to the left of the entry $\pi_i$. It is easy to see
that every sequence $e_1e_2\ldots e_n$ with the property that $e_i \in [0,i-1]$ for all $i$
corresponds uniquely to a permutation.  Such a sequence is called an
\emph{inversion sequence}.

A \emph{vincular pattern} is a sequence $p = p_1 p_2 \ldots p_r$, where
some disjoint subsequences of two or more adjacent entries may be
underlined, satisfying $p_i \in [0,r-1]$ for each $i$, where
any value $j > 0$ can only appear in $p$ if $j -1$ appears as well.
The \emph{standardization} of a word $w = w_1w_2 \ldots w_k$ is the word
obtained by replacing all instances of the $i$th smallest entry of $w$
with $i$.

An inversion sequence $e$ \emph{avoids} the vincular pattern $p$ if
there is no subsequence $e_{i_1} e_{i_2}\ldots e_{i_r}$ of $e$ whose
standardization is $p$, and such that $i_{s+1} = i_{s} +1$ whenever
$p_{s}$ and $p_{s+1}$ are part of the same underlined
subsequence. $\mathcal{I}_n(p)$ denotes the set of inversion sequences
of size $n$ that avoid the vincular pattern $p$ and $I_n(p)$ denotes
the size of this set.

Auli and Elizalde~\cite{elizalde} showed that
$I_n(\underline{10}0)= I_n(\underline{10}1)$. The known numbers of this sequence
\cite[A336070]{OEIS} are
\[1, 1 ,2 ,6 ,23 ,106 ,567, 3440, 23286, 173704, 1414102.
\]
Given a sequence of integers $x=x_1\dots x_n$, we say
that there is a \emph{descent} at position $i$ if $x_i>x_{i+1}$; we
denote by
\[
  D(x) = \{i: x_i>x_{i+1}\}\quad\text{and}\quad
  \desbot(x) = \{x_{i+1}: i\in D(x)\}
\]
the set of descents and \emph{descent bottoms}, respectively.
The main result of this section is the following theorem.

\begin{theorem}\label{weak_ascent_elizalde}
  The number of length-$n$ weak ascent sequences is the same as the
  number of length-$n$ inversion sequences that avoid the vincular
  pattern $\underline{10}0$. Further, there is a descent preserving
  bijection between the two sets.
\end{theorem}

We prove Theorem~\ref{weak_ascent_elizalde} by establishing a
bijection $\phi$ between the sets $\Wasc_n$ and
$\mathcal{I}_n(\underline{10}0)$. First, we recall a crucial property
of the elements of $\mathcal{I}_n(\underline{10}0)$ from Auli and
Elizalde~\cite{elizalde}: For each $e=e_1\dots e_n\in \mathcal{I}_n(\underline{10}0)$,
\[
  e_n\in [0, n-1]\setminus \desbot(\bar{e}),
\]
where $\bar{e}=e_1\dots e_{n-1}$ denotes the inversion sequence that
we obtain by deleting the last entry of $e$.  Another important
observation is that the descent bottoms of $e$ are distinct elements.
Indeed, if $e_i$ and $e_j$, with $i<j$, are descent bottoms and
$e_i=e_j$, then the subword $e_{i-1}e_{i}e_j$ would form a
$\underline{10}0$ pattern.

On the other hand, for any weak ascent sequence
$w=w_1\dots w_n$ we have, by definition,  $w_n\in [0, \wasc(\bar{w})+1]$,
and $\wasc(\bar{w})= |\bar{w}| - 1 - |D(\bar{w})|$. That is,
\[
  w_n\in [0, n-1-|D(\bar{w})|].
\]

\begin{proof}[Proof of Theorem \ref{weak_ascent_elizalde}]
  Given $w=w_1\dots w_n$ in $\Wasc_n$ and $e=e_1\dots e_{n}$ in
  $\mathcal{I}_n(\underline{10}0)$, let $\bar{w}=w_1\dots w_{n-1}$
  and $\bar{e} = e_1\dots e_{n-1}$ (as above) so that we can write
  $w=\bar{w}w_n$ and $e=\bar{e}e_n$.  We shall define
  $\phi:\Wasc_n\to\mathcal{I}_n(\underline{10}0)$ using recursion. The
  base case, $n=0$, is that $\phi$ maps the empty word to the empty
  word. Assume $n\geq 1$. Let $\phi(\bar{w}w_n) = \bar{e}e_n$, where
  $\bar{e} = \phi(\bar{w})$ and $e_n$ is the $w_n$th element of
  $[0, n-1]\setminus \desbot(\bar{e})$ when listed in increasing
  order. In other words, if
  $[0, n-1]\setminus \desbot(\bar{e}) = \{\ell_0,\dots,\ell_{k}\}$
  and $\ell_0 < \cdots <\ell_{k}$, then $e_n = \ell_i$ with $i = w_n$.
  We claim that
  \begin{enumerate}
  \item\label{Claim1} $\phi$ is a well defined mapping,
  \item\label{Claim2} $\phi$ is bijective, and
  \item\label{Claim3} $\phi$ preserves descent; i.e.\ if $\phi(w)=e$, then $D(w)=D(e)$.
  \end{enumerate}
  Assume that $w\in \Wasc_n$ and $\phi(w)=e$. By definition we have
  $\phi(\bar{w})=\bar{e}$, and by induction we may further assume that
  $D(\bar{w})=D(\bar{e})$. We now prove each of the three claims
  separately.

  \textit{Proof of \eqref{Claim1}}. The only thing that can go wrong
  in the definition of $\phi$ is that the $w_n$th element of
  $[0, n-1]\setminus \desbot(\bar{e})$ may, a priori, not
  exist. Recall that, for $e\in \mathcal{I}(\underline{10}0)$, the
  descent bottoms of $e$ are distinct. Thus,
  $|\desbot(\bar{e})| = |D(\bar{e})|$. Further,
  $D(\bar{e}) = D(\bar{w})$ (by induction) and thus there are as
  many elements in $[0, n-1]\setminus \desbot(\bar{e})$ as there are
  in $[0, n-1-|D(\bar{w})|]$, which is the set that $w_n$ belongs to.

  \textit{Proof of \eqref{Claim2}}. Assume that
  $\phi(u) = \phi(v) = e$. Then $\phi(\bar{u}) = \phi(\bar{v})$ and,
  by induction, $\bar{u} = \bar{v}$. Also, if
  $[0, n-1]\setminus \desbot(\bar{e}) = \{\ell_0,\dots,\ell_{k}\}$ and
  $e_n = \ell_i$, then $u_n = i = v_n$. Thus $\phi$ is injective. To
  see that $\phi$ is surjective, let
  $e\in \mathcal{I}_n(\underline{10}0)$ be given. By induction there
  is a $\bar{w}$ in $\Wasc_{n-1}$ such that
  $\phi(\bar{w})=\bar{e}$. Further, with $e_n = \ell_i$, we let
  $w_n=i$. Then $\phi(w)=e$.

  \textit{Proof of \eqref{Claim3}}.  By the induction hypothesis we
  have $D(\bar{w}) = D(\bar{e})$. What remains to show is that
  $n-1\in D(w)$ if and only if $n-1\in D(e)$. Note that $e_{n-1}$ is a
  member of $[0,n-1]\setminus \desbot(\bar{\bar{e}})$; let this set be
  $\{\ell_0,\ldots, \ell_r\}$. Further, let $w_{n-1} = i$, so that
  $\phi$ maps $w_{n-1}$ to $e_{n-1} = \ell_i$. There are two
  possibilities: If $e_{n-1}$ is a descent bottom, then $e_n$ belongs
  to the set
  $[0,n]\setminus \{\desbot(\bar{\bar{e}})\cup\{e_{n-1}\}\} =
  \{\ell_0, \ell_1,\ldots, \ell_{i-1},\ell_{i+1}, \ldots, \ell_r,
  \ell_{r+1}\}$, where $\ell_{r+1}=n$. Let $j=w_n$. If $j<i$ then
  $n-1$ is both in $D(w)$ and $D(e)$.  If $j\geq i$, then
  $w_{n-1}\leq w_n$ and $n-1\not\in D(w)$, but then also
  $e_n \geq \ell_{i+1}>\ell_i$, and hence $n-1\not\in D(e)$. If $e_{n-1}$
  is not a descent bottom, then $e_n$ is from the set
  $[0,n]\setminus \desbot(\bar{\bar{e}}) = \{\ell_0, \ell_1,\ldots,
  \ell_r, \ell_{r+1}\}$, where $\ell_{r+1}=n$. In this case the
  statement is clear.
\end{proof}

\begin{example}
  With $w = 010101$ we find that $e = \phi(w) = 010213$ as detailed in
  the following table:
  \[
    \begin{array}{l|l|l|l}
      n & w & e & [0,n]\setminus\desbot(e)\\[0.5ex]\hline
      &&\\[-2ex]
      0 & \emptyword & \emptyword & 0 \\
      1 & 0          & 0          & 0,1 \\
      2 & 01         & 01         & 0,1,2 \\
      3 & 010        & 010        & 1,2,3 \\
      4 & 0101       & 0102       & 1,2,3,4\\
      5 & 01010      & 01021      & 2,3,4,5\\
      6 & 010101     & 010213     & 2,3,4,5,6\\
    \end{array}
  \]
  E.g.\ when constructing the row for $n=4$ we are given $w=0101$ and
  $\bar{e}=010$ (from the preceding row). The value of $e_4$ is
  $\ell_{w_4}=\ell_{1}$, where $(\ell_0,\ell_1,\ell_2)=(1,2,3)$ are the
  values in the last column of the preceding row. That is, $e_4=2$.
\end{example}

We give another set of inversion sequences that are also in one-to-one
correspondence with weak ascent sequences.  Let
$\mathcal{I}_n(D)$ denote the sequences of integers
$w = w_1\dots w_n$ with
$$w_i\in [0, i-1]\setminus D(w_1\dots w_{i-1}),$$
i.e., the set of length-$n$ inversion sequences where the positions of
descents are forbidden as entries.  Another way to describe
this is as the set of inversion sequences that contain only entries at
which positions a weak ascent occur.  The sequence $0102$ is not in
$\mathcal{I}_4(D)$, because $w_2$ is a descent top, hence
the value $2$ is forbidden.  All other length-4 inversion sequences
are in $\mathcal{I}_4(D)$.

While we do not give a formal proof of the following result, let us
note that it follows by an argument similar to that given for
Theorem~\ref{weak_ascent_elizalde} in conjunction with the bijection
$\Lambda$ from partition matrices to inversion sequences that was
given in \cite{partition_and_composition_matrices}.

\begin{proposition}
  The set $\mathcal{I}_n(D)$ is equinumerous with the set
  $\Wasc_n$.
\end{proposition}

Note that the inductive construction is very similar in each case,
weak-ascent sequences, inversion sequences avoiding the vincular
pattern, and the weak Fishburn permutations. In each case there is a
set of possible values for the $j$th entry that is determined in the
prefix of length $j-1$.  Auli and Elizalde~\cite{elizalde} use the
method of generating trees to derive an expression for the generating
function
\begin{align*}
 A(z)=\sum_{n\geq 0} I_n(\underline{10}0)z^n =\sum_{n\geq 0}I_n(\underline{10}1)z^n.
\end{align*}

\begin{proposition}[\cite{elizalde} Proposition 3.12]
  We have that $A(z) =G(1,z)$, where $G(u,z)$ is defined recursively
  by
  \begin{align*}
    G(u,z) = u(1-u)+uG(u(1+z-uz),z).
  \end{align*}
\end{proposition}

This expression and the bijection from the proof of
Theorem~\ref{weak_ascent_elizalde} imply that if we denote by $A_n$
the enumeration sequence that counts the number of weak ascent
sequences of length $n$, we have $A_n = \sum_{k=0}^n a_{n,k}$, where
$a_{n,k}$ is given by the following formula. The initial values
$a_{0,0} = 1$, $a_{n,0} = a_{0,k} = 0$ and
\begin{align}\label{a_n_k}
  a_{n,k} = \sum_{i=0}^n \sum_{j=0}^{k-1} (-1)^{j} \binom{k-j}{i}\binom{i}{j} a_{n-i,k-j-1}.
\end{align}

\begin{proposition}
  The number of weak ascent sequences of length $n$ having $k$ weak
  ascents is $a_{n,k+1}$.
\end{proposition}

\begin{proof}
  Let $W_{n,k}$ denote the set of weak ascent sequences with $k$ weak
  ascents. After the last weak ascent in a weak ascent sequence  the
  entries are in decreasing order. Let us
  consider the objects where some of these descents are marked. More
  precisely, let $W_{n,k}^{(a)}$ be the set of weak ascent
  sequences of length $n$ with $k$ weak ascents, where there are $a$
  marked descents in the last maximal decreasing subsequence. For
  instance, $0012012115\underline{3}2\underline{0}$ is an element of
  $W_{13,7}^{(2)}$, because the last decreasing subsequence is
  $(5,3,2,0)$, from which $3$ and $0$ is marked. Let further
  $S_0=W_{n,k-1}\cup W_{n,k-2}\cup W_{n,k-2}^{(1)}$ and for
  $1\leq j\leq k-1$
  \begin{align*}
    S_j= W_{n,k-j-1}^{(j-1)}\cup W_{n,k-j-1}^{(j)}\cup W_{n,k-j-2}^{(j)}\cup W_{n,k-j-2}^{(j+1)}.
  \end{align*}

  It is clear then that
  \begin{align}\label{w_n,k}
    |W_{n,k-1}| = |S_0|-|S_1|+|S_2|-\cdots+(-1)^{k-1}|S_{k-1}|.
  \end{align}
  In order to enumerate the sets $S_j$, we describe the way
  of its elements are constructed.
  For a given $i$, $0\leq i\leq n$, consider the set of weak ascent sequences of length
  $n-i$ with $k-2$ weak ascents $W_{n-i,k-2}$.
  We want to augment a sequence from the set $W_{n-i,k-2}$ by a
  decreasing sequence of $i$ elements in order to obtain a valid weak
  ascent sequence of length $n$. Because of the definition of a weak
  ascent sequence, these elements are from the set
  $\{0,1,\ldots, k-1\}$, so we can choose from the $k$ possible values
  $i$ and attach to the underlying sequence in decreasing
  order. However, two cases can happen. If the greatest chosen value
  is greater than or equal to the $(n-i)$th entry, then we obtain a
  weak ascent sequence with $k-1$ weak ascents, hence an object from
  $W_{n,k-1}$ such that the last weak ascent is at the $(n-i+1)$th
  position. So, letting $i$ go from $1$ to $n$ we generated all the
  objects in $W_{n,k-1}$ exactly once.

  On the other hand, if the greatest chosen value is smaller than the
  $(n-i)$th entry, then we obtain a weak ascent sequence with $k-2$
  weak ascents. In this case we mark the $(n-i+1)$th entry, which is a
  descent in the last decreasing subsequence. Letting $i$ go from $1$
  to $n$, we obtain this way every object exactly once from the set
  $W_{n,k-2}^{(1)}$.

  Consider now the general set $S_j$. For $i\geq j$, take a weak
  ascent sequence of length $n-i$ with $k-j-2$ weak ascents. Choose
  from the $k-j$ available values $i$ and attach them in decreasing
  order to obtain a valid weak ascent sequence. After, choose from the
  $i$ values $j$ out, mark those that are descents. (It can happen
  that the first value is chosen and it is a weak ascent, in this case
  do not mark this entry.)

  If $i=j$, we do not have any free choice in this last case ($j$
  elements out of $i$), but there are two possibilities:
  \begin{itemize}
  \item[I.0] The greatest chosen value is greater than or equal to the
    $(n-i)$th entry. Then we obtained a weak ascent sequence with
    $k-j-1$ weak ascents with $j-1$ marks on all the descents of the
    last decreasing subsequence.
  \item[II.0] The greatest chosen value is smaller than the $(n-i)$th
    entry. Then we obtained a weak ascent sequence with $k-j-2$ weak
    ascents and $j$ marks on the last decreasing subsequence.
  \end{itemize}
  For $i>j$ there are four cases:
  \begin{itemize}
  \item[I.] The greatest chosen value is greater than or equal to the
    $(n-i)$th entry.
    \begin{itemize}
    \item[1.] If this greatest value was chosen among the $j$ (out of
      the $i$) then we obtain a weak ascent sequence with $k-j-1$ weak
      ascents and $j-1$ marked descents in the last decreasing
      subsequence.
    \item[2.] If this greatest value was not chosen among the $j$ (out
      of the $i$) then we obtain a weak ascent sequence with $k-j-1$
      weak ascents and $j$ marked descents in the last decreasing
      subsequence.
    \end{itemize}
  \item[II.] The greatest chosen value is smaller than the $(n-i)$th
    entry.
    \begin{itemize}
    \item[1.] If the greatest value was chosen among the $j$ then we
      obtain a weak ascent sequence with $k-j-2$ weak ascents and $j$
      marked descents in the last decreasing subsequence.
    \item[2.] If the greatest value was not chosen among the $j$ then
      mark this descent also. This way we obtain a weak ascent
      sequence with $k-j-2$ weak ascents and $j+1$ marked descents in
      the last decreasing subsequence.
    \end{itemize}
  \end{itemize}
  Letting $i$ go from $i=j$ to $n$ we have that in the cases I.0 and
  I.1 together all the objects in $W_{n,k-j-1}^{(j-1)}$ are
  constructed exactly once. Similarly, during the above construction
  we get in the case I.2 the objects in $W_{n,k-j-1}^{(j)}$, in the
  cases II.0 and II.1 the objects in $W_{n,k-j-2}^{(j)}$ and in the
  case II.2 the objects $W_{n,k-j-2}^{j+1}$.

  The counting formula for the size of the sets $S_j$ is
  straightforward from the construction:
  \begin{align}\label{s_j}
    |S_{j}| = \sum_{i=j}^{n} \binom{k-j}{i}\binom{i}{j}|W_{n-i,k-j-2}|.
  \end{align}
  and hence, by Equation (\ref{w_n,k}) we have
  \begin{align*}
    |W_{n,k-1}| = \sum_{j=0}^{k-1} (-1)^{j}  \sum_{i=j}^{n} \binom{k-j}{i}\binom{i}{j}|W_{n-i,k-j-2}|.
  \end{align*}
  We see that $|W_{n,k-1}|$ satisfies the same recurrence relation as
  $a_{n,k}$ in Equation (\ref{a_n_k}).
\end{proof}

Note that $a_{n,n}$ are the Catalan numbers, which is clear, since
weak ascent sequences that have only ascents are in a trivial
bijection for instance with Dyck paths.

\begin{remark}
  Since the sequence $A_n$ has a rapid growth, greater than
  $n^{{n}/{2}}$, the series $\sum_{n=0}^{\infty}A_nz^n$ converges
  only for $z=0$.  On the other hand, since $A_n \leq n!$, the
  exponential generating function
  $\sum_{n=0}^{\infty} A_n \frac{t^n}{n!}$ determines an analytic
  function on a certain domain. However, it could be difficult to
  represent it as a function by using classical functions. We did not
  manage to derive a nice closed formula for it.
\end{remark}

\mynewpage
\section{Concluding remarks}
Experimentation with restricted classes of weak ascent sequences has
shown that there are relationships to other known number sequences.
As an example, we offer the following simple Catalan result:
\begin{proposition}
  The number of weak ascent sequences $w=(w_1,\ldots,w_n)$ that are
  weakly-increasing, i.e. $ w_{i}\leq w_{i+1}$ for all $i$, is given
  by the Catalan numbers.
\end{proposition}
\begin{proof}
 If a weak-ascent sequence is weakly
    increasing then there is no restriction on the entries, so the set
    of weakly-increasing weak ascent sequences is the same as the set
    of nondecreasing sequences of integers $a_i$ with
    $0\leq a_i\leq i$ which are known to be enumerated by the Catalan
    numbers.
\end{proof}
A slightly different restriction gives rise to the following conjecture:
\begin{conjecture}
  The number of weak ascent sequences $w=(w_1,\ldots,w_n)$ that satisfy
  $w_{i+1} \geq w_i -1$ for all $i$ equals
  OEIS~\cite[A279567]{OEIS} ``Number of length $n$ inversion sequences
  avoiding the patterns 100, 110, 120, and 210.''
\end{conjecture}
The paper \cite{McNamara} probed restrictions on ascent sequences and
how such restrictions played out in the bijective correspondences.
The above proposition and conjecture represent a first step in that
direction for weak ascent sequences.

Research into pattern avoidance in ascent sequences (see Duncan and
Steingr\'{i}msson~\cite{Duncan}) proved to be a fruitful avenue of
research that produced a wealth of enumerative identities and
conjectures, some of which are still open.  The asymptotics of
generating functions for these has recently been investigated by
Conway et al.~\cite{conway}.  We posit that a similarly rich
collection of results are to be discovered by exploring pattern
avoidance for weak ascent sequences
and that equidistribution results hold for multivarite statistics
on weak ascent sequences in the same way they were shown to hold for
ascent sequences~\cite{Schlosser}.

\mynewpage
\section*{Acknowledgments}
The authors would like to thank Toshiki Matsusaka for his assistance
with Equation~\eqref{a_n_k}.

\end{document}